\documentclass[10pt, a4paper, oneside, reqno]{amsproc}
\usepackage{mathtools}
\usepackage{amsthm}
\usepackage{amsfonts}
\usepackage{amsmath}
\usepackage{amssymb}
\usepackage{mathrsfs}
\usepackage{euscript}   
\usepackage{color}
\usepackage{mathbbol}
\usepackage{pifont}
\usepackage{tikz}
\usepackage{blkarray}
\usepackage{hyperref}
\usepackage{url}

\numberwithin{equation}{section}
\newtheorem{theorem}{Theorem}[section]
\newtheorem{proposition}[theorem]{Proposition}
\newtheorem{corollary}[theorem]{Corollary}
\newtheorem{lemma}[theorem]{Lemma}
\newtheorem{qst}[theorem]{Question}

\newtheorem*{opq*}{\bf Problem}
\newtheorem*{SDT}{\bf Stinespring Dilation Theorem}
\newtheorem*{PropA}{\bf Proposition~A}
\newtheorem{definition}[theorem]{Definition}

\theoremstyle{remark}
\newtheorem{remark}[theorem]{Remark}

\theoremstyle{definition}
\newtheorem{example}[theorem]{Example}

\newcommand*{\D}{\mathrm{d\hspace{.1ex}}}
\newcommand*{\Ge}{\geqslant}

\newcommand*{\Le}{\leqslant}

\begin{document}
   \title[Sum of self-commutators]{Sum of self-commutators of commuting operators}
\author[S. Chavan]{Sameer Chavan}

\address{Department of Mathematics, IIT Kanpur, Uttar Pradesh, India -208016}

\email[]{chavan@iitk.ac.in}

\author[Md. R. Reza]{Md. Ramiz Reza}

\address{Department of Mathematics and Statistics, Aliah University, Kolkata -700160}

\email[]{ramiz.math@aliah.ac.in}

\author[S. S. Sequiera]{Shanola S. Sequeira}

\address{Department of Mathematics, IIT Kanpur, Uttar Pradesh, India -208016}

\email[]{shanolasequeira@gmail.com, shanolas@iitk.ac.in}

\thanks{The first author is supported by the Advanced Research Grant $\mathsf{ANRF/ARG/2025/001228MS}$}


   \begin{abstract}
It is well known that every hyponormal operator on a complex Hilbert space whose spectrum has planar Lebesgue measure zero is normal. In particular, every compact hyponormal operator on a complex Hilbert space is normal. In this paper, we investigate analogous structural and spectral phenomena for so-called sum-hyponormal 
$d$-tuples on a complex Hilbert space $\mathcal H$, namely, commuting $d$-tuples ${\bf T}=(T_1, \ldots, T_d)$ satisfying
$\sum_{j=1}^d [T^*_j, T_j] \Ge 0$. 
We show that every sum-hyponormal $d$-tuple of compact operators decomposes as the direct sum of a normal $d$-tuple and a quasinilpotent sum-hyponormal $d$-tuple. As a consequence, every sum-hyponormal $d$-tuple on a finite-dimensional Hilbert space is normal. 
In contrast to the single-operator case, a sum-hyponormal $d$-tuple need not be normaloid.
Nevertheless, 
we establish a multivariable analogue of Putnam's inequality for sum-hyponormal $d$-tuples with commuting imaginary parts.
As a consequence, we prove that ${\bf T}$ is normal whenever the planar Lebesgue measure of $s(\sigma({\bf T}))$ is zero, where $\sigma({\bf T})$ denotes the Taylor spectrum of ${\bf T}$ and ${s}$ is the complex-linear polynomial with alternating coefficients $1$ and $i$.
\end{abstract}
   \maketitle

\section{Introduction and preliminaries}
This paper is motivated by Putnam's inequality (see \cite[Theorem~1]{P1970}), which asserts that the norm of the self-commutator of a hyponormal operator is bounded above by the normalized area measure of its spectrum. This striking structural result has played an important role in the development of the theory of hyponormal operators (see \cite{MP1989}). 
One of its most notable consequences is that every hyponormal operator is normal whenever the area of its spectrum is zero.
In particular, every compact hyponormal operator on a complex Hilbert space is normal (see \cite[Corollary~III.1.6]{MP1989}).
Thus, the theory of hyponormal operators is inherently infinite-dimensional. A key ingredient in the direct proof of this last assertion is the equality of the spectral radius and the operator norm for hyponormal operators (see \cite[Corollary~III.1.4]{MP1989}). An abstract analogue of this equality holds in the setting of completely positive maps. Since this generalization is not needed in the main body of the paper, we have relegated its proof to the appendix.

Before discussing higher-dimensional counterparts of Putnam's inequality and related results, we briefly introduce some notation and recall relevant notions.
Let $\mathbb Z_+$ and $\mathbb N$ denote the sets of nonnegative and positive integers, respectively. Let $\mathbb R$ and $\mathbb C$ denote the fields of real and complex numbers, respectively. For a positive integer $d$, let $X^d$ denote the $d$-fold Cartesian product of a nonempty set $X.$ The permutation group on $\{1, \ldots, d\}$ is denoted by $\mathfrak{S}_d$. 
The group $\mathfrak{S}_d$ acts on $X^d$ via
$$\eta \cdot (x_1, \ldots, x_d)  = (x_{\eta^{-1}(1)}, \ldots, x_{\eta^{-1}(d)}), \quad \eta \in \mathfrak{S}_d, ~(x_1, \ldots, x_d)\in X^d.$$ 
For each $j=1, \ldots, d,$ let $\varepsilon_j$ denote the element of $\mathbb Z^d_+$ with $1$ in the $j$th position and $0$ elsewhere. Let $\lfloor \cdot \rfloor$ denote the greatest integer function and let $\bmod$ denote the modulo operation. Let $\mathbb C[x_1, \ldots, x_d]$ denote the ring of polynomials in the variables $x_1, \ldots, x_d$ with complex coefficients. For $z = (z_1, \ldots, z_d) \in \mathbb C^d,$ we write $\bar{z}=(\bar{z}_1, \ldots, \bar{z}_d)$, 
where $\bar{w}$ denotes the complex conjugate of $w \in \mathbb C$. Let $\mathcal H$ be a complex Hilbert space, and let $\dim(\mathcal H)$ denote its Hilbert space dimension. 
For a  subset $E$ of $\mathcal H$, $\vee E$ denotes the closed linear span of $E$, with the convention that $\vee \, \emptyset=\{0\}$. For a closed subspace $\mathcal M$ of $\mathcal H$, we denote its orthogonal complement by $\mathcal H \ominus \mathcal M$. Let $\mathcal B(\mathcal H)$ denote the unital $C^*$-algebra of bounded linear operators on $\mathcal H$, with the identity operator $I$ as its unit and the adjoint operation $A \mapsto A^*$ as its involution.
Let ${\mathcal C}(\mathcal H)$ denote the norm-closed ideal of compact operators on $\mathcal H$. For $A, B \in \mathcal B(\mathcal H)$, we write $[A,B]=AB-BA$ for the commutator of $A$ and $B$.
We denote the real and imaginary parts of $A$ by $\Re(A)$ and $\Im(A)$, respectively, where
\begin{equation}
\label{real-imaginary}
\Re(A)=\frac{A+A^*}{2}, \quad \Im(A)=\frac{A-A^*}{2i}.
\end{equation} 

Let $d$ be a positive integer. A {\it commuting $d$-tuple on $\mathcal H$} is a $d$-tuple ${\bf T}=(T_1,\ldots,T_d)$ of pairwise commuting operators in $\mathcal B(\mathcal H)$; that is, $[T_j, T_k]=0$ for every $1 \Le j, k \Le d$. We denote its adjoint by ${\bf T}^*=(T^*_1, \ldots, T^*_d)$. For any commuting $d$-tuple ${\bf T}$, the following identities hold:
\begin{align}
\label{d-comm-formula}
 \big[T_j, [T^*_k, T_k]\big] = \big[[T_j, T^*_k], T_k\big],
\quad 1 \Le j, k \Le d, \\
\label{d-comm-formula-im}
 \big[\Im (T_j), \Im (T_k)\big] =\frac{i}{2}\Im\big([T^*_j, T_k]\big), 
\quad 1 \Le j, k \Le d.
\end{align}
The second-order commutation relation \eqref{d-comm-formula} is a special case of the commutator form of the Jacobi identity (see \cite[Theorem~1.1.4]{DK2000}), whereas \eqref{d-comm-formula-im} follows by expanding the commutator and using \eqref{real-imaginary}. For $\lambda=(\lambda_1, \ldots, \lambda_d) \in \mathbb C^d,$ the $d$-tuple $(T_1-\lambda_1 I, \ldots, T_d-\lambda_d I)$ is denoted by ${\bf T}-\lambda$. For $\alpha =(\alpha_1, \ldots, \alpha_d) \in \mathbb Z^d_+,$  we write ${\bf T}^{\alpha}=\prod_{j=1}^d T^{\alpha_j}_j$. We denote the joint kernel of $T_1, \ldots, T_d$ by $\ker({\bf T})=\cap_{j=1}^d \ker(T_j)$. For an invariant subspace $\mathcal M$ of ${\bf T},$ we write ${\bf T}|_{\mathcal M}$ for the $d$-tuple $(T_1|_{\mathcal M}, \ldots, T_d|_{\mathcal M})$. The notations $\sigma_p({\bf T}),$ $\sigma_{ap}({\bf T}),$ $\sigma({\bf T}),$ $\sigma_e({\bf T}),$ $r({\bf T})$ denote the point spectrum, approximate point spectrum, Taylor spectrum, essential spectrum, spectral radius of ${\bf T}$, respectively (see \cite{B1971, Cu1981, T1970} for the relevant definitions). 

An operator $A \in \mathcal B(\mathcal H)$ is said to be {\it nilpotent} if there exists a positive integer $k$ such that $A^k=0$. The smallest positive integer $k$ with this property is called the {\it nilpotency index} of $A$. 
A commuting $d$-tuple ${\bf T}$ is said to be {\it quasinilpotent} if $\sigma({\bf T})=\{0\}.$ 
We say that ${\bf T}=(T_1,\ldots,T_d)$ is {\it normal} (resp. {\it essentially normal}) if $[T_j^*,T_j]=0$ (resp. $[T_j^*,T_j] \in \mathcal C(\mathcal H)$) for every $j=1,\ldots,d$.  Furthermore, ${\bf T}$ is {\it completely nonnormal} if there is no nonzero subspace that reduces ${\bf T}$ to a normal tuple.
A commuting $d$-tuple ${\bf T}$ is said to be {\it doubly commuting} if $[T^*_j, T_k]=0$ for every $1 \Le j \neq k \Le d.$ 
Consider the completely positive map $\varphi_{\bf T}$ induced by ${\bf T}$:
\begin{equation} 
\label{varphi-T}
\varphi_{\bf T}(X) = \sum_{j=1}^d T^*_jXT_j, \quad X \in \mathcal B(\mathcal H).
\end{equation}
We say that ${\bf T}$ is {\it normaloid} if $r({\bf T})=\|\varphi_{\bf T}(I)\|^{1/2}$ (cf. \cite{CT1984, S2025}).
Following \cite{At1988}, ${\bf T}$ is said to be {\it hyponormal} if the {\it commutator matrix} $$[\![{\bf T}^*, {\bf T}]\!]:=\displaystyle \big([T_k^*,T_j]\big)_{1\Le j,k \Le d}$$ is positive semidefinite on the $d$-fold orthogonal direct sum of $\mathcal H$. Any hyponormal $d$-tuple is normaloid (see \cite[Lemma~3.10]{CS2017}). In fact,
the proof of \cite[Lemma~3.10]{CS2017} shows that
\begin{equation} \label{special-appendix}
\varphi_{\bf T}(I)^2 \Le \varphi^2_{\bf T}(I)\Longrightarrow r({\bf T})
=
\|\varphi_{\bf T}(I)\|^{1/2}.
\end{equation}
(See Proposition~A for a generalization.)
We refer the reader to \cite{At1988, B1971, CS2017, CCHZ2000, CT1984, Cu1990, DY1992, EP2001, HK2010, M2002, MPS2022, S2025, X1983} for multivariable generalizations of hyponormality and related concepts. 

We now introduce the following notion of hyponormality, which forms the main focus of this paper.
\begin{definition} Let ${\bf T}=(T_1, \ldots, T_d)$ be a commuting $d$-tuple on a complex Hilbert space $\mathcal H$. The {\it commutator} of ${\bf T}$ is defined by
\begin{equation}
\label{joint-commutator}
[{\bf T^*}, {\bf T}] := \sum_{j=1}^d  [T^*_j, T_j].
\end{equation}
We say that 
${\bf T}$ is {\it sum-normal} if $[{\bf T^*}, {\bf T}] = 0,$ and {\it sum-hyponormal} if $[{\bf T^*}, {\bf T}] \Ge 0.$
\end{definition}
\begin{remark} \label{rem-1} 
Let $T$ be a commuting $d$-tuple on $\mathcal H.$ 
\begin{enumerate}
\item[(a)] ${\bf T}$ is sum-normal if and only if both ${\bf T}$ and ${\bf T^*}$ are sum-hyponormal.
\item[(b)] 
If ${\bf T}$ is sum-hyponormal (resp. sum-normal), then ${\bf T}|_{\mathcal M}$ is sum-hyponormal (resp. sum-normal) for every invariant subspace $\mathcal M$ of ${\bf T}$. 
\item[(c)] The notions of sum-hyponormality and sum-normality are invariant under the action of the permutation group. Indeed, for any $\eta \in {\mathfrak S}_d$, 
\begin{equation*} 
[(\eta \cdot {\bf T})^*, \eta \cdot {\bf T}]=[{\bf T}^*, {\bf T}]. 
\end{equation*}
\end{enumerate}
The commutator $[{\bf T}^*, {\bf T}]$ defined in \eqref{joint-commutator} may be viewed as the {\it trace} of the operator matrix $[\![{\bf T}^*, {\bf T}]\!]$. It is worth mentioning that a notion of the determinant of $[\![{\bf T}^*, {\bf T}]\!]$ is introduced and investigated in \cite{MPS2022}.
\hfill $\diamondsuit$
\end{remark}

The present paper focuses on the following questions.
\begin{qst} \label{Q1}
Let $\mathcal H$ be a complex Hilbert space. 
\begin{enumerate}
\item[(i)]  
Is every sum-hyponormal $d$-tuple of compact operators on $\mathcal H$ necessarily sum-normal or normal?
\item[(ii)] Is every sum-normal $d$-tuple on $\mathcal H$ necessarily normal? 
\item[(iii)] Under what additional conditions is a sum-hyponormal $d$-tuple $($resp. sum-normal $d$-tuple$)$ on $\mathcal H$ necessarily hyponormal (resp. normal$)$\!?
\end{enumerate}
\end{qst}

The Drury--Arveson $d$-shift (see \cite{H2023}) is a $d$-tuple of commuting hyponormal operators and hence is sum-hyponormal. 
By modifying the first finitely many weights of the Drury--Arveson 2-shift, one can construct a sum-hyponormal unilateral weighted $2$-shift ${\bf W}=(W_1, W_2)$ such that neither of $W_1$ and $W_2$ is hyponormal (see Example~\ref{q-hypo}). This example shows that the compactness assumption in Question~\ref{Q1}(i) cannot be omitted. It also shows that a sum-hyponormal tuple is, in general, not normaloid. Consequently, the approach in \cite{A1963, MP1989, S1962} cannot be used to address Question~\ref{Q1}(i).

\section{Main results}

In this section, we state the main results of this paper. We begin with the following theorem, which reveals the structure of sum-normal tuples with finite-dimensional coinvariant subspaces.
\begin{theorem} \label{main-2}
Let ${\bf T}$ be a commuting $d$-tuple on a complex Hilbert space $\mathcal H$. Let $\mathcal M$ be an invariant subspace of ${\bf T}$ such that $\dim(\mathcal H \ominus \mathcal M) < \infty$.  
If ${\bf T}$ is sum-normal, then
$\mathcal M$ reduces ${\bf T}$, ${\bf T}|_{\mathcal M}$ is sum-normal, and ${\bf T}|_{\mathcal H \ominus \mathcal M}$ is normal.
\end{theorem}

Applying Theorem~\ref{main-2} with $\mathcal M=\{0\}$, we conclude that every sum-normal tuple on a finite-dimensional Hilbert space is normal. Since every operator on a finite-dimensional Hilbert space is trace class and every sum-hyponormal tuple of trace-class operators is sum-normal (see Remark~\ref{cyclicity-trace}), every sum-hyponormal tuple on a finite-dimensional Hilbert space is also normal (see Proposition~\ref{main-1}).

To address Question~\ref{Q1}(i) in full generality, we next establish a decomposition theorem for sum-hyponormal tuples consisting of compact operators (see \cite[Corollary~4.2]{EP2001} for a decomposition theorem for arbitrary commuting tuples).

\begin{theorem} \label{main-4}
Let ${\bf T}$ be a sum-hyponormal $d$-tuple of compact operators on a complex Hilbert space $\mathcal H$. 
Then the closed subspace 
\begin{equation} 
\label{ortho-M}
\mathcal M = 
\bigvee_{{\lambda} \in  \sigma_p({\bf T})}\ker({\bf T}- {\lambda}) 
\end{equation}
reduces ${\bf T}$. Moreover,  
${\bf T}$ admits the decomposition
\begin{equation*}
{\bf T}={\bf N} \oplus {\bf S}~\mbox{on~}\mathcal H = \mathcal M \oplus (\mathcal H \ominus \mathcal M),
\end{equation*}
where ${\bf N}$ is a normal $d$-tuple on $\mathcal M$ and ${\bf S}$ is a completely nonnormal, quasinilpotent sum-hyponormal $d$-tuple on $\mathcal H \ominus \mathcal M$. 
Either summand may be absent.
\end{theorem}

Theorem~\ref{main-4} shows that Question~\ref{Q1}(i) has an affirmative answer provided every quasinilpotent sum-normal $d$-tuple is the zero $d$-tuple ${\bf 0}$. 
This naturally leads to the question of whether a sum-normal $d$-tuple consisting entirely of nilpotent operators must be zero. The following result answers this question in the affirmative.

\begin{theorem} \label{nilpotent-thm}
Let ${\bf T}$ be a sum-hyponormal $d$-tuple on $\mathcal H$ consisting of nilpotent operators.  
Then, ${\bf T}={\bf 0}$.
\end{theorem}

The previous result raises the natural question of whether the tuple of nilpotent operators can be replaced by a quasinilpotent tuple. Before addressing this question, we establish a multivariable analogue of Putnam's inequality (see \cite[Theorem~1]{P1970}; for other multivariable analogues of Putnam's inequality, see \cite[Theorem~1.1]{CCHZ2000}, \cite[Theorem~3.2]{M2002}, and \cite[Theorem~5]{X1983}). We begin by introducing the notation. Let $s \in \mathbb C[x_1, \ldots, x_d]$ be the polynomial defined by
\begin{equation}
\label{s-polynomial}
{s}(x_1, \ldots, x_d):=\sum_{j=1}^{\lfloor d/2 \rfloor}(x_{2j-1}+ix_{2j})+  (d \bmod 2)x_{d}.
\end{equation}  
Note that $s$ is a surjective complex-linear polynomial. In particular, if $E \subseteq \mathbb C$ satisfies $\mathrm{m}_2(E)=0$, then $\mathrm{m}_{2d}(s^{-1}(E))=0$, where $\mathrm{m}_{2d}$ denotes the Lebesgue measure on $\mathbb C^d$.
\begin{theorem} \label{new-thm-Putnam}
Let ${\bf T}=(T_1, \ldots, T_d)$ be a sum-hyponormal $d$-tuple on a complex Hilbert space $\mathcal H$ satisfying
\begin{equation}
\label{Im-commute-new}
[\Im(T_j), \Im(T_k)]=0, \quad 1 \Le j, k \Le d.
\end{equation} 
Then, for every $\eta \in \mathfrak S_d$, 
\begin{equation}
\label{Putnam-d-var}
\|[{\bf T}^*, {\bf T}]\| \Le \frac{1}{\pi}\, \mathrm{m}_2\big(\sigma\big(s(\eta \cdot {\bf T})\big)\big).
\end{equation}
In particular, if there exists a subset $E \subseteq \mathbb C$ of planar Lebesgue measure zero such that $\sigma({\bf T}) \subseteq s^{-1}(E)$, then ${\bf T}$ is normal.
\end{theorem}

In the one-variable case, the additional assumption \eqref{Im-commute-new} in Theorem~\ref{new-thm-Putnam} is vacuous, so the theorem is an exact generalization of Putnam's inequality.
Moreover, it shows that there are no nonzero quasinilpotent sum-hyponormal $d$-tuples satisfying \eqref{Im-commute-new} (corresponding to the case $E=\{0\}$). A key ingredient in the proof is the fact that every sum-normal $d$-tuple satisfying \eqref{Im-commute-new} is necessarily normal (see Lemma~\ref{coro-Putnam} and Proposition~\ref{main-3-coro}). It is worth noting that the additional assumption \eqref{Im-commute-new} is satisfied by every doubly commuting $d$-tuple (see \eqref{d-comm-formula-im}). Furthermore, there exist commuting pairs of isometries satisfying \eqref{Im-commute-new} that are not doubly commuting (see, for example, \cite[Theorem~1]{II2006} and \cite[Theorem~2]{IO1994}). Finally, \eqref{Putnam-d-var} is 
nontrivially optimal (see \eqref{optimal}).

We now state the final main result of this paper answering Question~\ref{Q1}(iii).
\begin{theorem} \label{main-3}
Let ${\bf T}=(T_1, \ldots, T_d)$ be a sum-hyponormal $d$-tuple on a complex Hilbert space $\mathcal H$ satisfying
\begin{equation}
\label{double-commutator}
\big[T_j, [T^*_k, T_k]\big] =0,  
\quad 1 \Le j \neq k \Le d.
\end{equation}
Then $T_1, \ldots, T_d$ are hyponormal operators. In particular, every doubly commuting sum-hyponormal $d$-tuple on a complex Hilbert space is hyponormal.
\end{theorem}

The proofs of Theorems~\ref{main-2}--\ref{main-3} are presented in Section~\ref{Sect2}. In Section~\ref{Sect3}, we discuss several consequences of these results; see Propositions~\ref{main-1}, and \ref{mixed-cohypo-coro}--\ref{sum-normal-to-normal}, as well as Corollary~\ref{coro-single}. 

\section{Proofs of the main results \label{Sect2}}

\begin{proof}[Proof of Theorem~\ref{main-2}]
Let $\mathcal H_1$ be an invariant subspace of a commuting $d$-tuple ${\bf T}$ on a Hilbert space $\mathcal H$. With respect to the orthogonal decomposition $\mathcal H= \mathcal H_1 \oplus \mathcal H_2,$ the operators $T_j$ admit the block matrix representations: 
\begin{equation} \label{matrix-rep}
T_j=\begin{pmatrix}
A_j & X_j \\
0 & B_j
\end{pmatrix}, \quad j=1,\ldots,d.
\end{equation}
Since $[T_j, T_k]=0,$ for $1 \Le j, k \Le d,$ 
\begin{equation} \label{T-commute}
[A_j, A_k] =  0,~
[B_j, B_k] = 0, \quad
 j, k = 1, \ldots, d.
\end{equation}
It is straightforward to see using \eqref{matrix-rep} that 
\begin{align*}
T_j^* T_j = 
\begin{pmatrix}
A_j^* A_j  & A_j^* X_j  \\[4pt]
X_j^* A_j & X_j^* X_j+ B_j^* B_j
\end{pmatrix},
\quad
T_jT_j^*=\begin{pmatrix}
A_j A_j^* +  X_j  X_j^* &  X_j B_j^* \\[4pt]
B_j X_j^* & B_j B_j^*
\end{pmatrix} \\
\Longrightarrow [T_j^*, T_j] = 
\begin{pmatrix}
[A_j^*, A_j] - X_jX_j^*  & A_j^*X_j-X_jB_j^*  \\[4pt]
X_j^*A_j-B_jX_j^* & X_j^* X_j+ [B_j^*, B_j]
\end{pmatrix}, \quad j=1,\ldots,d.
\end{align*}
By the sum-normality of ${\bf T},$ we have 
\begin{align}\label{0-identity-0}
[{\bf A}^*,{\bf A}]= \sum\limits_{j=1}^d X_jX_j^*,\\
\label{1-identity-1}
[{\bf B}^*,{\bf B}]= -\sum\limits_{j=1}^d X_j^*X_j,
\end{align}
where ${\bf A}=(A_1, \ldots, A_d)$ and ${\bf B}=(B_1, \ldots, B_d)$. 
Assume further that $\dim(\mathcal H_2) < \infty.$ 
Since ${\bf B}$ is a commuting  $d$-tuple on $\mathcal H_2$ (see \eqref{T-commute}), it follows that ${{\bf B}^*}$ admits an eigenvector $v$ of unit norm corresponding to an eigenvalue $(\lambda_1,\ldots, \lambda_d)\in \mathbb C^d$:
\begin{equation} \label{joint-eigen new}
B_j^*v = \lambda_j v, \quad j=1,2,\ldots, d.
\end{equation}
This, together with \eqref{1-identity-1}, implies that
\begin{equation} \label{identity 3 new}
\sum_{j=1}^d(\|B_j v\|^2-|\lambda_j|^2) 
+ \sum_{j=1}^d \|X_j v\|^2 = 0. 
\end{equation}
On the other hand,  by \eqref{joint-eigen new}, 
\begin{equation} \label{eigen-eqn new}
\langle B_jv, v\rangle = \bar{\lambda}_j, \quad j=1,\ldots,d.
\end{equation}  
Hence, by the Cauchy--Bunyakovsky--Schwarz inequality (see \cite[Theorem~I.1.4]{Co1990}), 
$|\lambda| \Le \|B_jv\|$ for $j=1,\ldots,d.$
Together with \eqref{identity 3 new}, this implies that
\begin{align} \label{XAv-XBv-0}
X_j v = 0, \quad 1 \Le j \Le d, \\ \label{XAv-XBv-1}
\|B_j v\| = |\lambda_j|, \quad 1 \Le j \Le d.
\end{align} 
By \eqref{eigen-eqn new} and \eqref{XAv-XBv-1}, $$|\langle B_jv, v\rangle|=\|B_j v\|\|v\|, \quad 1\Le j \Le d.$$  
By the equality case of the Cauchy--Bunyakovsky--Schwarz inequality and using \eqref{eigen-eqn new}, we have $B_jv=\bar{\lambda}_jv $ for each $j=1,\ldots,d.$ Hence, the one-dimensional subspace spanned by $\{v\}$ reduces ${\bf B}$. 
As $\dim(\mathcal H_2) < \infty$, an induction argument shows that  the $d$-tuple ${\bf B}$ is normal. Furthermore, 
by \eqref{XAv-XBv-0}, 
\begin{equation} \label{XA-XB=0-new}
X_j = 0, \quad 1 \Le j \Le d.
\end{equation}
Consequently, $\mathcal H_1$ is a reducing subspace of ${\bf T}.$ Moreover, by \eqref{0-identity-0} and \eqref{XA-XB=0-new}, the $d$-tuple ${\bf A}$ is sum-normal.
\end{proof}
\begin{remark} \label{restriction-reducing}
Let ${\bf T}$ be a sum-hyponormal $d$-tuple and let $\mathcal M$ be an invariant subspace of ${\bf T}$. An examination of the proof of Theorem~\ref{main-2} shows that if ${\bf A}:={\bf T}|_\mathcal M$ is sum-normal, then $\mathcal M$ reduces ${\bf T}$. 
Indeed, decomposing each $T_j$ as in \eqref{matrix-rep} and arguing as in the proof of Theorem~\ref{main-2}, we obtain $\sum_{j=1}^d  X_jX^*_j \Le [{\bf A^*}, {\bf A}]=0$ (cf. \eqref{0-identity-0}). 
Hence, $X_j =0$ for every $1 \Le j \Le d$. Therefore, $\mathcal M$ reduces ${\bf T}$. \hfill $\diamondsuit$
\end{remark}

We now provide a proof of Theorem~\ref{main-4}. This is based on the following lemma.
\begin{lemma} 
\label{eigen-adjoint}
Let ${\bf T}$ be a commuting $d$-tuple on $\mathcal H$. Then, for every $\lambda \in \mathbb C^d$, the following statements hold$:$
\begin{enumerate}
\item[(i)] If ${\bf T}$ is sum-hyponormal, then 
$\ker({\bf T} - {\lambda}) \subseteq\ker({\bf T}^*-\bar{\lambda})$.
\item[(ii)] If ${\bf T}$ is sum-normal, then $\ker({\bf T}-\lambda)=\ker({\bf T}^*-\bar{\lambda})$.
\end{enumerate}
\end{lemma}
\begin{proof}
Note that, for any commuting $d$-tuple ${\bf T}$ on $\mathcal H$,
\begin{equation}\label{T-xI} 
[{\bf T}^*- \bar{\lambda}, {\bf T}-\lambda] =  [{\bf T}^*, {\bf T}], \quad \lambda \in \mathbb C^d. 
\end{equation} 
Moreover, since $\ker {\bf T} \subseteq \ker {\bf T}^*$ holds for any sum-hyponormal $d$-tuple ${\bf S}$, 
the formula \eqref{T-xI} yields (i). Assertion (ii) then follows from (i) and Remark~\ref{rem-1}(a). 
\end{proof}
\begin{remark} 
If ${\bf T}$ is a sum-hyponormal $d$-tuple, then by Lemma~\ref{eigen-adjoint}(i), $$\sigma_p({\bf T}) \subseteq \big\{\lambda \in \mathbb C^d : \bar{\lambda} \in \sigma_p({\bf T}^*)\big\}.$$ 
Thus, if ${\bf T}$ is completely nonnormal, then $\sigma_p({\bf T})=\emptyset$ (cf. \cite[Proposition~III.1.1]{MP1989}). Lemma~\ref{eigen-adjoint}(ii) implies that any sum-normal $d$-tuple has quasitriangular property in the sense of \cite[Definition~3]{HK2010}. \hfill $\diamondsuit$
\end{remark}

\begin{proof}[Proof of Theorem~\ref{main-4}] Let $\lambda, \mu \in \sigma_p({\bf T})$ with $\lambda \neq \mu$. Thus $\lambda_j \neq \mu_j$ for some $1 \Le j \Le d$. Let $x, y \in \mathcal H$ be nonzero vectors such that $T_jx=\lambda_j x$ and $T_jy=\mu_j y$. By Lemma~\ref{eigen-adjoint}(i),  
\begin{equation*}
\lambda_j \langle x, y \rangle= \langle T_j x, y \rangle = \langle x, T^*_j y\rangle = \mu_j \langle x, y \rangle,  
\end{equation*}
which implies that $\langle x, y \rangle=0$. 
Thus, $\mathcal M$ is given by
\begin{equation*}
\mathcal M = \bigoplus_{{\lambda} \in  \sigma_p({\bf T})} \ker({\bf T}- {\lambda}),
\end{equation*}
and $\mathcal M$ reduces ${\bf T}$ to a normal $d$-tuple. If $\sigma_p({\bf T})=\emptyset$, then $\mathcal M = \{0\}$, and the first summand is absent. If $\mathcal H$ is finite dimensional, then $\mathcal M=\mathcal H$, and the second summand is absent. Therefore, we may assume that $\mathcal H$ is infinite dimensional. 
By \cite[Theorem 3.1]{B1996}, every nonzero point in the Taylor spectrum $\sigma({\bf T})$ of ${\bf T}$ is an eigenvalue with corresponding eigenvector belonging to $\mathcal M$. 
Also, since
\begin{equation*}
\sigma({\bf T})=\sigma({\bf T}|_{\mathcal M}) \cup \sigma({\bf T}|_{\mathcal H \ominus \mathcal M})
\end{equation*} 
is a union of two nonempty sets (see \cite[Corollary, p.~185]{T1970}), it follows that $\sigma({\bf T}|_{\mathcal H \ominus \mathcal M})=\{{0}\}$, that is, ${\bf T}|_{\mathcal H \ominus \mathcal M}$ is quasinilpotent. 
Since any normal $d$-tuple ${\bf N}$ is normaloid 
(see \cite[Lemma~3.10]{CS2017}; see also Proposition~A), 
this implies that ${\bf S}:={\bf T}|_{\mathcal H \ominus \mathcal M}$ is completely nonnormal. 
\end{proof}
\begin{remark}
Let ${\bf T}$ be a commuting $d$-tuple of compact operators on a separable Hilbert space $\mathcal H$. 
By \cite[Theorems~7.1.9 and 7.2.1]{RR2000}, 
there is a chain 
$\mathcal M_0 \subset \mathcal M_1 \subset \mathcal M_2 \subset \cdots$
of closed subspaces of $\mathcal H$ such that 
\begin{itemize}
		\item[(i)] for every integer $j \Ge 0,$ $\mathcal M_j$ is an invariant subspace of ${\bf T}$,   
		\item[(ii)] for every integer $j \Ge 0,$ $\dim(\mathcal M_{j+1} \ominus \mathcal M_j) < \infty$, 
		\item[(iii)] ${\bigcup_{j=0}^\infty \mathcal M_j}$ is dense in $\mathcal H.$
	\end{itemize}
Assume that ${\bf T}$ is sum-normal. By 
Remark~\ref{rem-1}(b), for each integer $j \Ge 0,$ we apply Theorem~\ref{main-2} to ${\bf T}|_{\mathcal M_{j+1}}$ with $$\mathcal M_{j+1} = \mathcal M_j \oplus (\mathcal M_{j+1} \ominus \mathcal M_j)$$ to conclude that
$\mathcal M_j$ is a reducing subspace of ${\bf T}|_{\mathcal M_{j+1}}$, ${\bf T}|_{\mathcal M_j}$ is sum-normal, and ${\bf T}|_{\mathcal M_{j+1} \ominus \mathcal M_j}$ is normal. By the condition (iii),   
\begin{align*} \mathcal H = \displaystyle \bigoplus_{i \Ge 0} \mathcal M_i \ominus \mathcal M_{i-1},
\end{align*}
where $\mathcal M_{-1}=\{0\}$. With respect to this decomposition, ${\bf T}$ decomposes as
\[
{\bf T}=
\begin{pmatrix}
{\bf T}|_{\mathcal M_0} & 0 & 0 & \cdots \\
0 & {\bf T}|_{\mathcal M_1 \ominus \mathcal M_0} & 0 & \cdots \\
0 & 0 & {\bf T}|_{\mathcal M_2 \ominus \mathcal M_1} & \cdots \\
\vdots & \vdots & \vdots & \ddots
\end{pmatrix},
\]
where ${\bf T}|_{\mathcal M_0}$ is sum-normal and ${\bf T}|_{\mathcal M_{j+1} \ominus \mathcal M_j}$ is normal for every $j \Ge 0.$ Let ${\lambda} \in \sigma({\bf T}|_{\mathcal M_0})\setminus \{{0}\}.$ Then, by \cite[Theorem 3.1]{B1996}, ${\lambda}$ is an eigenvalue of ${\bf T}|_{\mathcal M_0}$. This, combined with Remark~\ref{restriction-reducing}, shows that $\ker({\bf T}-{\lambda})$ reduces ${\bf T}|_{\mathcal M_0}$. It follows that ${\bf T}|_{\mathcal M}$ is sum-normal with ${\lambda} \notin \sigma({\bf T}|_{\mathcal M}),$ where $\mathcal M := \mathcal M_0 \ominus \ker({\bf T}-{\lambda}).$ Continuing this inductively, we obtain a special case of Theorem~\ref{main-4}. \hfill $\diamondsuit$
\end{remark}

We now proceed to the proof of Theorem~\ref{nilpotent-thm}. The proof relies on the following cancellation lemma.
\begin{lemma} \label{prod-lem}
Let ${\bf A}=(A_1, \ldots, A_d)$ be a sum-hyponormal $d$-tuple on a complex Hilbert space $\mathcal H$. If, for some nonzero $\alpha \in \mathbb Z^d_+$, 
\begin{equation}
\label{assumption-cancellation}
{\bf A}^{\alpha + \varepsilon_l}=0, \quad l=1, \ldots, d,
\end{equation}
then ${\bf A}^\alpha=0$.
\end{lemma}
\begin{proof} 
By \eqref{assumption-cancellation} and the sum-hyponormality of ${\bf A}$, 
\begin{align*} -{\bf A}^{*\alpha}\Big(\sum_{l=1}^d A_l A^*_l \Big){\bf A}^{\alpha} &= {\bf A}^{*\alpha}[{\bf A}^*, {\bf A}]{\bf A}^{\alpha} \Ge 0.
\end{align*}
It follows that 
$A^*_l {\bf A}^{\alpha} = 0$ for $l =1, \ldots, d$.
Since $\alpha$ is nonzero, there exists $1 \Le l_0 \Le d$ such that $\alpha_{l_0} \Ge 1$. As ${\bf A}$ is commuting, we have $A^*_{l_0}A_{l_0}{\bf A}^{\alpha - \varepsilon_{l_0}}=0$.
Thus $${\bf A}^{\alpha - \varepsilon_{l_0}}(\mathcal H) \subseteq \ker(A^*_{l_0}A_{l_0})=\ker(A_{l_0}),$$ and hence ${\bf A}^{\alpha}=0.$
\end{proof}

\begin{proof}[Proof of Theorem~\ref{nilpotent-thm}]
Assume that $T_j$ is a nilpotent operator of nilpotency index $k_j$ for each $j=1, \ldots, d$. 
Thus
\begin{equation} \label{n-index}
T^{k_j}_j =0, ~T^{k_j-1}_j \neq 0, \quad j=1, \ldots, d. 
\end{equation}
If, for some $1 \Le j \Le d,$ $k_j=1,$ then $T_j=0.$ Deleting  $T_j$ from ${\bf T}$ yields a sum-hyponormal $(d-1)$-tuple consisting of nilpotent operators. Hence, without loss of generality, we may assume that $k_j\Ge 2$ for every $j=1, \ldots, d$. 

Let $I=\{1, \ldots, d\}$.
The commutativity of ${\bf T}$, together with \eqref{n-index}, implies that 
$$T_l \prod_{j \in I} T^{k_j-1}_j=0, \quad l \in I.$$ 
Hence, by Lemma~\ref{prod-lem}, 
\begin{equation} 
\label{Step-J1-multi}
\prod_{j \in I} T^{k_j-1}_j=0.
\end{equation}
We verify by strong finite induction $k$ that
\begin{equation} 
\label{Step-J1-multi-claim}
T^{k_{i_1}-k}_{i_1}\prod_{j \in I\setminus\{i_1\}} T^{k_j-1}_j=0, \quad k =1, \ldots, k_{i_1}, ~i_1 \in I.
\end{equation}
By \eqref{Step-J1-multi}, this holds for $k=1$. If \eqref{Step-J1-multi-claim} holds for $k=1, \ldots, l$ for some $l \Le k_{i_1}-1,$ then by \eqref{n-index}, 
\begin{equation*}
T_m T^{k_{i_1}-l-1}_{i_1}\prod_{j \in I\setminus\{i_1\}} T^{k_j-1}_j=0, \quad m=1, \ldots, d.
\end{equation*}  
Another application of Lemma~\ref{prod-lem} completes the verification of \eqref{Step-J1-multi-claim}. We next verify by strong finite induction $k$ that
\begin{equation} 
\label{Step-J1-multi-claim-2nd}
T^{k_{i_1}-2}_{i_1} T^{k_{i_2}-k}_{i_2}  \prod_{j \in I\setminus\{i_1, i_2\}} T^{k_j-1}_j=0, \quad k =1, \ldots, k_{i_2}, ~i_1 \neq i_2 \in I.
\end{equation}
By \eqref{Step-J1-multi-claim}, this holds for $k=1$. If \eqref{Step-J1-multi-claim-2nd} holds for $k=1, \ldots, l$ for some $l \Le k_{i_2}-1,$ then  by \eqref{n-index}, 
\begin{equation*}
T_m T^{k_{i_1}-2}_{i_1} T^{k_{i_2}-l-1}_{i_2}\prod_{j \in I\setminus\{i_1, i_2\}} T^{k_j-1}_j=0, \quad m=1, \ldots, d.
\end{equation*}  
Another application of Lemma~\ref{prod-lem} completes the verification of \eqref{Step-J1-multi-claim-2nd}. A similar argument with \eqref{Step-J1-multi-claim} replaced by \eqref{Step-J1-multi-claim-2nd} yields  
\begin{equation*} 
T^{k_{i_1}-3}_{i_1} T^{k_{i_2}-k}_{i_2}  \prod_{j \in I\setminus\{i_1, i_2\}} T^{k_j-1}_j=0, \quad k =1, \ldots, k_{i_2}.
\end{equation*}
We continue this to obtain
\begin{equation*} 
T^{k_{i_1}-j}_{i_1} T^{k_{i_2}-k}_{i_2}  \prod_{j \in I\setminus\{i_1, i_2\}} T^{k_j-1}_j=0, \quad j=1, \ldots, k_{i_1}, ~k =1, \ldots, k_{i_2}.
\end{equation*}
Continuing this, for any $1 \Le i_1, \ldots,  i_{d-1} \Le d,$ we obtain
$$\prod_{l \in I\setminus\{i_1, i_2, \ldots, i_{d-1}\}} T^{k_l-1}_l=0.$$ 
This contradicts \eqref{n-index}, completing the proof. 
\end{proof}

We now turn to the proof of Theorem~\ref{new-thm-Putnam}. We begin by recalling a useful identity required in the proof. Let $A, B, C, D \in \mathcal B(\mathcal H)$. Then 
\begin{align} 
\label{formula-Cartesian}
[(A+iB)^*,C+iD]
=[A^*,C]+[B^*,D]
+i\big([A^*,D]-[B^*,C]\big).
\end{align}
\begin{lemma} \label{coro-Putnam}
Let $d=2^k$ for some positive integer $k$. Let ${\bf T}=(T_1, \ldots, T_{d})$ be a sum-normal $d$-tuple on a complex Hilbert space $\mathcal H$ such that 
\begin{equation} \label{im-commute}
[T^*_j, T_k]=[T^*_k, T_j], \quad 1 \Le j, k \Le d. 
\end{equation}
Then ${\bf T}$ is normal. 
\end{lemma}
\begin{proof} 
We prove the desired conclusion by induction on $k \Ge 1.$ 
First, suppose that $k=1$, that is, $d=2$.
Let $N=T_1+iT_2$. By \eqref{formula-Cartesian},
\begin{align*} 
[N^*,N]
&=[T_1^*,T_1]+[T_2^*,T_2]
+i\big([T_1^*,T_2]-[T_2^*,T_1]\big).
\end{align*}
Since ${\bf T}$ is sum-normal, by \eqref{im-commute}, 
$N$ is a normal operator.
Also, since $[T_1, T_2]=0$, 
$$[T_j,N]=0, \quad j=1,2.$$
As $N$ is normal, the Fuglede--Putnam theorem (see \cite[Corollary~1.18]{RR1973}) yields
$$[T_j,N^*]=0, \quad j=1,2.$$
Consequently,
$[T_1^*,T_j]
=i[T_2^*,T_j]$, $j=1, 2$. Since $[T_j^*,T_k]$ is a self-adjoint operator for every $1 \Le j, k \Le 2$ (see \eqref{im-commute}), it follows that $[T_j^*,T_j]=0,$ $j=1, 2$.
This completes the proof of the base case $k=1$.

Assume that the result holds for all sum-normal $d$-tuples satisfying \eqref{im-commute} when $d=2^k$. We prove that any $d$-tuple ${\bf T}$ satisfying \eqref{im-commute} is normal when $d=2^{k+1}$. Relabel the operators $T_1, \ldots, T_{2^{k+1}}$ as $$T_{1,1}, T_{2,1}, T_{1,2}, T_{2,2}, \ldots, T_{1,2^{k}}, T_{2,2^{k}}.$$
Let $1 \Le j \Le 2^{k}$, $N_j=T_{1,j}+iT_{2,j}$. Then
${\bf N}=(N_1, \ldots, N_{2^{k}})$ is a commuting $2^k$-tuple. 
Another application of \eqref{formula-Cartesian} yields
\begin{align} 
\notag
[N_j^*,N_l]
&=[T_{1,j}^*,T_{1,l}]
+[T_{2,j}^*,T_{2,l}] \\ \label{N1-commutator-second-final}
&+i\big([T_{1,j}^*,T_{2,l}]
-[T_{2,j}^*,T_{1,l}]\big), ~ 1 \Le j, l \Le 2^{k}.
\end{align}
This, together with the sum-normality of ${\bf T}$ and \eqref{im-commute}, shows that ${\bf N}$ is a sum-normal $2^{k}$-tuple. 
Combining \eqref{im-commute} and \eqref{N1-commutator-second-final}, we obtain
\begin{equation*} 
[N^*_j, N_l]=[N^*_l, N_j], \quad 1 \Le j, l \Le 2^{k}. 
\end{equation*}
Applying the induction hypothesis for $d=2^{k}$, we conclude that ${\bf N}$ is normal. This, together with \eqref{im-commute} and \eqref{N1-commutator-second-final}, implies that $(T_{1,j}, T_{2,j})$ is a sum-normal $2$-tuple for every $1 \Le j \Le 2^k$. In view of \eqref{im-commute}, an application of the base case $k=1$ shows that 
${\bf T}$ is normal.
\end{proof}

The next lemma establishes Putnam's inequality for $2^k$-tuples. 
\begin{lemma} \label{Putnam-prop}
Let $d=2^k$ for some positive integer $k$.
Let ${\bf T}=(T_1, \ldots, T_d)$ be a sum-hyponormal $d$-tuple on a complex Hilbert space $\mathcal H$ satisfying \eqref{im-commute}. Then
\begin{equation}
\label{Putnam-d-var-special}
\|[{\bf T}^*, {\bf T}]\| \Le \frac{1}{\pi}\, \mathrm{m}\big(\sigma\big(T_{1}+iT_{2}+ \cdots + T_{d-1}+iT_{d}\big)\big).
\end{equation}
In particular, if, $\sigma\big(T_{1}+iT_{2}+ \cdots + T_{d-1}+iT_{d}\big)$ has planar Lebesgue measure zero, then ${\bf T}$ is normal.
\end{lemma}
\begin{proof} 
The idea of the proof is similar to that of Lemma~\ref{coro-Putnam}. We prove the desired inequality by induction on $k \Ge 1.$ First, suppose that $k=1,$ that is, $d=2$.
By \eqref{formula-Cartesian} and \eqref{im-commute}, 
\begin{equation}
\label{Berger-Shaw-application}
[(T_1+iT_2)^*, T_1+iT_2] = [{\bf T}^*, {\bf T}] \Ge 0.
\end{equation}
Therefore, $T_1+iT_2$ is a hyponormal operator. Applying Putnam's inequality (see \cite[Theorem~VI.2.1]{MP1989}) yields \eqref{Putnam-d-var-special} for $d=2$. This completes the proof of the base case $k=1$.

Assume that \eqref{Putnam-d-var-special} holds for all sum-hyponormal $2^k$-tuples satisfying \eqref{im-commute}. 
We prove that any sum-hyponormal $d$-tuple ${\bf T}$ satisfying \eqref{im-commute} satisfies \eqref{Putnam-d-var-special} when $d=2^{k+1}$.
Let ${\bf T}$ be a sum-hyponormal $2^{k+1}$-tuple satisfying \eqref{im-commute}.
For $1 \Le j \Le 2^k$, $N_j=T_{2j-1}+iT_{2j}$.
Then
${\bf N}=(N_1, \ldots, N_{2^{k}})$ is a commuting $2^k$-tuple. 
Another application of \eqref{formula-Cartesian} yields
\begin{align} 
\notag
[N_j^*,N_l]
&=[T_{2j-1}^*,T_{2l-1}]
+[T_{2j}^*,T_{2l}] \\ \label{N1-commutator-second-final-new}
&+i\big([T_{2j-1}^*,T_{2l}]
-[T_{2j}^*,T_{2l-1}]\big), ~ 1 \Le j, l \Le 2^{k}.
\end{align}
This, together with the sum-hyponormality of ${\bf T}$ and \eqref{im-commute}, shows that ${\bf N}$ is a sum-hyponormal $2^{k}$-tuple. Indeed,
\begin{equation}
\label{equal-commutator}
[{\bf N}^*, {\bf N}] = [{\bf T}^*, {\bf T}].
\end{equation} 
Combining \eqref{im-commute} and \eqref{N1-commutator-second-final-new}, we obtain
\begin{equation*} 
[N^*_j, N_l]=[N^*_l, N_j], \quad 1 \Le j, l \Le 2^{k}. 
\end{equation*}
Applying the induction hypothesis for $d=2^{k}$, we conclude from \eqref{equal-commutator} that 
\begin{align} \label{step-Putnam}
 \|[{\bf T}^*, {\bf T}]\| 
&= \|[{\bf N}^*, {\bf N}]\| \\ \notag
& \Le \frac{1}{\pi}\,  \mathrm{m}\big(\sigma\big(N_{1}+iN_{2}+ \cdots + N_{d-1}+iN_{d}\big)\big). 
\end{align}
However, for any $j=1, \ldots, d-1$,
\begin{align*}
N_{j}+iN_{j+1} 
= T_{2j-1}+iT_{2j}+ iT_{2j+1}-T_{2(j+1)}.
\end{align*}
Since the sum-hyponormality and \eqref{im-commute} is invariant under the change $T_j \mapsto -T_j,$ 
this, together with \eqref{step-Putnam}, yields
\begin{align*}
 \|[{\bf T}^*, {\bf T}]\| 
 \Le \frac{1}{\pi}\,  \mathrm{m}\big(\sigma\big(T_{1}+iT_{2}+ iT_3 + T_4 + \cdots + iT_{d-1}+T_{d}\big)\big). 
\end{align*}
Since the sum-hyponormality and \eqref{im-commute} is invariant under the action of $\mathfrak S_d$ (see Remark~\ref{rem-1}(c)), this completes the proof of \eqref{Putnam-d-var-special}.

Finally, if,$\sigma\big(T_{1}+iT_{2}+ \cdots + T_{d-1}+iT_{d}\big)$ has planar Lebesgue measure zero, then \eqref{Putnam-d-var-special} implies that ${\bf T}$ is sum-normal. Hence, by Lemma~\ref{coro-Putnam}, ${\bf T}$ is normal.
\end{proof}
\begin{remark}
If ${\bf T}$ is a sum-hyponormal $2$-tuple satisfying \eqref{im-commute}, then it follows from \eqref{Berger-Shaw-application} and the Berger-Shaw inequality (see \cite[Corollary~VI.1.4]{MP1989}) that
$$\mbox{trace}[{\bf T}^*, {\bf T}] \Le \frac{1}{\pi}{\mathfrak m}(T_1+iT_2)\,{\mathrm m}(\sigma(T_1+iT_2)),$$ where ${\mathfrak m}(\cdot)$ denotes the rational multiplicity (see \cite[Definition~VI.1.1]{MP1989}). One limitation of this inequality is that, even when ${\bf T}$ is cyclic, the rational multiplicity ${\mathfrak m}(T_1+iT_2)$ may still be infinite (see \cite{DY1992, MPS2022} for further discussion of the Berger-Shaw phenomenon in several variables). \hfill $\diamondsuit$
\end{remark}
\begin{proof}[Proof of Theorem~\ref{new-thm-Putnam}]
In view of the identity \eqref{d-comm-formula-im}, 
the assumption \eqref{Im-commute-new} is equivalent to the condition \eqref{im-commute}.
Let ${\bf T}=(T_1, \ldots, T_d)$ be a sum-hyponormal $d$-tuple satisfying \eqref{Im-commute-new}. Choose a positive integer $k$ such that $d < 2^k$, and define 
\begin{equation}
\label{2k-to-general}
S_j = \begin{cases} T_j, & \mbox{if} ~1 \Le j \Le d, \\
0, & \mbox{if}~ d+1 \Le j \Le 2^k.
\end{cases}
\end{equation}
Then ${\bf S}=(S_1, \ldots, S_{2^k})$ is a sum-hyponormal $2^k$-tuple satisfying \eqref{Im-commute-new}. Hence, by Lemma~\ref{Putnam-prop} and the inclusion $\mathfrak{S}_d \hookrightarrow \mathfrak{S}_{2^k}$, we obtain \eqref{Putnam-d-var} for the identity permutation in $\mathfrak S_d$. This, combined with Remark~\ref{rem-1}(c), yields \eqref{Putnam-d-var} for any $\eta \in \mathfrak S_d$.  

To prove the remaining assertion, assume that there exists a set $E \subseteq \mathbb C$ of planar Lebesgue measure zero such that $\sigma({\bf T}) \subseteq s^{-1}(E)$
(see \eqref{s-polynomial}). By the spectral mapping theorem (see \cite[Theorem~4.8]{T1970b}), $\sigma(s({\bf T})) \subseteq E$. We may therefore apply Lemma~\ref{Putnam-prop} to conclude that ${\bf T}$ is normal, completing the proof.
\end{proof}

The proof of Theorem~\ref{main-3} relies on two auxiliary lemmas. We first establish a multivariable analogue of \cite[Corollary~1]{R1966} (cf.~\cite[Problem~236]{H1967}).
 
\begin{lemma}\label{comnotbb}
Let ${\bf T}$ be a commuting $d$-tuple on a nonzero complex Hilbert space $\mathcal H$. Then, there does not exist a real number $\mu > 0$
such that $[{\bf T}^*, {\bf T}] \Ge \mu I$.
\end{lemma}
\begin{proof}  
Assume that there exists a real number $\mu \Ge 0$ such that $[{\bf T}^*, {\bf T}] \Ge \mu I$. 
It is known that the approximate point spectrum of a commuting $d$-tuple on a nonzero complex Hilbert space is always nonempty (see \cite[Proposition~2]{B1971}).
Let $\lambda = (\lambda_1, \ldots, \lambda_d) \in \sigma_{ap}({\bf T}).$  
Thus there exists a sequence $\{h_n\}_{n \Ge 1}$ of unit vectors $h_n$ in $\mathcal H$ such that
\begin{equation}
\label{limit-app}
\lim_{n \rightarrow \infty}\sum_{j=1}^d \|(T_j-\lambda_j)h_n\|^2 =0.
\end{equation}
It follows from \eqref{T-xI} that $[{\bf T}^*- \bar{\lambda}, {\bf T}-\lambda] =  [{\bf T}^*, {\bf T}]\Ge \mu I,$ and hence 
\begin{equation*}
0 \Le \mu \Le \sum_{j=1}^d \|(T^{*}_j-\bar{\lambda}_j)h_n\|^2 + \mu \,\Le \, \sum_{j=1}^d \|(T_j-\lambda_j)h_n\|^2, \quad h \in \mathcal H,
\end{equation*}
and hence, by \eqref{limit-app}, $\mu =0$.
\end{proof}
\begin{remark}
\label{scalar-cpt}
Let $q : \mathcal B(\mathcal H) \rightarrow
B(\mathcal H)/ {\mathcal C}(\mathcal H)$ be the quotient 
map.
Since the Calkin algebra $\mathcal B(\mathcal H) / {\mathcal
C}(\mathcal H)$ is a unital $C^*$-algebra, there exist a Hilbert space $\mathcal K$ and
an injective unital $*$-representation $\pi : \mathcal B(\mathcal H) /
{\mathcal C}(\mathcal H) \rightarrow \mathcal B(\mathcal K)$ (see \cite[Chapter
VIII]{Co1990}). Consequently, the composition $\pi \circ q : \mathcal B(\mathcal H) \rightarrow \mathcal B(\mathcal
K)$ is a unital $*$-representation. 

Assume now that $[{\bf T}^*, {\bf T}]=\lambda I + K,$ where $\lambda \Ge 0$ and $K \in {\mathcal C}(\mathcal H)$.
Define $A_j = \pi \circ q(T_j),$ $j=1, \ldots, d,$ and let ${\bf A}=(A_1, \ldots, A_d)$ be a commuting $d$-tuple on $\mathcal K$. Since $[{\bf A}^*, {\bf A}]=\lambda I$, by Lemma~\ref{comnotbb}, $\lambda =0$. If $\lambda \Le 0,$ then $[{\bf T}, {\bf T}^*]=(-\lambda) I + (-K)$. Hence, by applying the preceding argument, we conclude that $\lambda =0.$
Therefore, 
$[{\bf T}^*, {\bf T}]$ cannot be expressed in the form $\lambda I + K,$ where $\lambda \neq 0$ and $K \in {\mathcal C}(\mathcal H)$. \hfill $\diamondsuit$
\end{remark}

The following lemma plays a key role in the proof of Theorem~\ref{main-3}.
\begin{lemma} \label{gen-fact}
Let ${\bf T}$ be a commuting $d$-tuple on a complex Hilbert space $\mathcal H$ and 
let $C$ be a self-adjoint operator in $\mathcal B(\mathcal H)$ such that $CT_j=T_jC,$ $j=1, \ldots, d$, and $[{\bf T}, {\bf T}^*] \Le C.$ Then $C \Ge 0$.  
\end{lemma}
\begin{proof} There is no loss of generality in assuming that $\mathcal H \neq \{0\}$.
By the spectral theorem for normal operators (see \cite[Theorem 1.12]{RR1973}), there exists a spectral measure $E$ supported on $\sigma(C)$ such that
	$$C=\int_{\sigma(C)} \lambda\, \D E(\lambda).$$ 
Since 
$CT_j=T_jC,$ by the spectral theorem (see \cite[Theorem 1.16]{RR1973}), $T_j$ commutes with $E(\sigma)$ for every Borel subset $\sigma$ of $\sigma(C)$ and every $j=1, \ldots, d$. In particular, $E(\sigma)\mathcal H$ reduces ${\bf T}$ for every Borel subset $\sigma$ of $\sigma(C)$. 
Let $\lambda_0 \in \sigma(C) \cap [-\|C\|, 0]$ and let $\mathcal M= E([-\|C\|, \lambda_0])\mathcal H$. Then, for every $x \in \mathcal M,$ 
\begin{align*}
\langle [{\bf T}, {\bf T}^*]x, \, x\rangle & \Le \langle Cx, x\rangle \\
&= \int_{[-\|C\|, \lambda_0]}
\lambda\, \D \langle E(\lambda)x,x\rangle \\
& \Le \lambda_0 \|x\|^{2}.
\end{align*}
Since $\mathcal M$ reduces ${\bf T}$, this implies that
$[{\bf T}|_{\mathcal M}, ({\bf T}|_{\mathcal M})^*]
\Le \lambda_0 I|_{\mathcal M},$ or equivalently, $$[({\bf T}|_{\mathcal M})^*, {\bf T}|_{\mathcal M}]
\Ge -\lambda_0 I|_{\mathcal M}.$$ This, combined with Lemma \ref{comnotbb}, implies that $\lambda_0 =0$. Thus, $\sigma(C) \cap [-\|C\|, 0] \subseteq \{0\},$ and hence $\sigma(C) \subseteq [0, \|C\|].$ Since $C$ is a self-adjoint operator, $C \Ge 0$. 
\end{proof}

\begin{proof}[Proof of Theorem~\ref{main-3}] 	 
Consider the $(d-1)$-tuple $\widehat{\bf T}_j$ obtained from ${\bf T}$ by deleting $T_j$, and let $C=[T^*_j, T_j],$ $j=1, \ldots, d.$ By \eqref{double-commutator}, we have $CT_k = T_k C$ for every $1 \Le k \neq j \Le d$. Hence, by Lemma~\ref{gen-fact}, each of the operators $T_1, \ldots, T_d$ is hyponormal. To see the remaining half, note that by \eqref{d-comm-formula},
\begin{equation} \label{application-d-comm}
 [T_j, T^*_k]=0 \Longrightarrow \big[T_j, [T^*_k, T_k]\big] = 0
\quad 1 \Le j \neq k \Le d. 
\end{equation}
Finally, a doubly commuting $d$-tuple consisting of hyponormal operators is hyponormal, since $[\![{\bf T}^*, {\bf T}]\!]$ is the diagonal matrix whose diagonal entries are $[T^*_j, T_j],$ $j=1, \ldots, d$. This, together with \eqref{application-d-comm}, 
proves the ``in particular” statement.
\end{proof}


\section{Consequences of the main results \label{Sect3}}

In this section, we present several consequences of the main results and their proofs given in Section~\ref{Sect2}. The first corollary of Theorem~\ref{main-2} can also be deduced from Theorem~\ref{main-4}.
\begin{proposition} \label{main-1}
Any sum-hyponormal $d$-tuple on a finite dimensional complex Hilbert space is normal.
\end{proposition}
\begin{proof} Let ${\bf T}$ be a sum-hyponormal $d$-tuple on a finite-dimensional complex Hilbert space $\mathcal H$. It is well known that if 
$A_1, \ldots, A_d \in {\mathcal B}(\mathcal H)$ are pairwise commuting operators, then there exists a nonzero vector $v \in \mathcal H$ and scalars $\lambda_1,\ldots,\lambda_d \in \mathbb{C}$ such that
$A_j v = \lambda_j v,$ $j=1, \ldots, d$ (see \cite{CT1979}).
By Theorem~\ref{main-4}, the closed subspace $\mathcal M$ of $\mathcal H$ defined by \eqref{ortho-M} reduces ${\bf T}$ to a normal $d$-tuple. Hence, by the aforementioned fact, $\mathcal H=\mathcal M,$ and consequently, ${\bf T}$ is normal. 
\end{proof}
\begin{remark} \label{cyclicity-trace}
Let ${\bf T}$ be a sum-hyponormal tuple of trace-class operators.
The cyclicity of the trace (see \cite[Theorem 3.6.7]{S2015}) shows that the trace of the positive operator $[{\bf T}^*, {\bf T}]$ is zero. Hence, ${\bf T}$ is sum-normal. \hfill $\diamondsuit$
\end{remark}

We now present an example of a non-normaloid, sum-hyponormal unilateral weighted $2$-shift showing that Proposition~\ref{main-1} no longer holds when $\mathcal H$ is infinite-dimensional.
\begin{example} \label{q-hypo} 
Let $\mathcal H$ be a separable complex Hilbert space with orthonormal basis 
$\{e_\alpha : \alpha \in \mathbb Z^2_+\}$. Consider the unilateral weighted $2$-shift ${\bf W}=(W_1, W_2)$, with weight multisequence $\big\{\omega^{(j)}_{\alpha} : \alpha \in \mathbb Z^2_+, ~j=1, 2\big\}$, given by 
\begin{align*}
\omega^{(1)}_{(0,1)}
&=\omega^{(2)}_{(1,0)}=\sqrt{\frac{1}{3}}, \quad
\omega^{(1)}_{(1,0)}
=\omega^{(2)}_{(0,1)}=\sqrt{\frac{1}{2}},\\
&\left.
\begin{aligned}
\omega^{(j)}_{(0,0)}
&=\sqrt{\frac{2}{3}}, \quad
\omega^{(j)}_{(1,1)}
=\sqrt{\frac{1}{2}},\\
\omega^{(j)}_{\alpha}
&=\sqrt{\frac{\alpha_j+1}{\alpha_1+\alpha_2+1}},
\quad
\alpha\in\mathbb Z_+^2\setminus\{(1,1)\},\ \alpha_1+\alpha_2 \Ge 2,
\end{aligned}
\right\},
\quad j=1,2.
\end{align*}
Recall that 
\begin{equation*}
W_j e_\alpha = w^{(j)}_\alpha e_{\alpha + \varepsilon_j}, \quad \alpha \in \mathbb Z^2_+, ~j=1, 2.
\end{equation*}
Since the weight multisequence
$\big\{{\omega^{(j)}_\alpha} : j
=1, 2, \,\alpha \in  \mathbb Z^2_+ \big\}$ is bounded by $1$, each of the operators $W_1$ and $W_2$ is bounded.
Also, $W_1W_2=W_2W_1$, since $$\omega^{(1)}_{(\alpha_1, \alpha_2)}\omega^{(2)}_{(\alpha_1+1, \alpha_2)}=\omega^{(2)}_{(\alpha_1, \alpha_2)}\omega^{(1)}_{(\alpha_1, \alpha_2+1)}, \quad (\alpha_1, \alpha_2) \in \mathbb Z^2_+.$$ 
For $j=1, 2,$ set $\omega^{(j)}_{\alpha}=0$ whenever $\alpha_1 < 0$ or $\alpha_2 < 0$. Then
\begin{align*}
\|W_je_\alpha\|^2- \|W^*_je_\alpha\|^2=(\omega^{(j)}_\alpha)^2- (\omega^{(j)}_{\alpha- \epsilon(j)})^2, \quad  \alpha \in \mathbb Z^2_+,\, j=1, 2.
\end{align*}  
Consequently, for every $\alpha = (\alpha_1, \alpha_2) \in\mathbb Z^2_+$, 
\begin{align*}
\sum_{j=1}^2 \Big(\|W_je_\alpha\|^2- \|W^*_je_\alpha\|^2\Big)
=\begin{cases} 4/3, & \mbox{if}~\alpha=(0, 0),\\
1/6, & \mbox{if}~\alpha=(1, 0), (0, 1), \\
5/6, & \mbox{if}~\alpha=(2, 0), (0, 2),\\
5/12, & \mbox{if}~\alpha=(2, 1), (1, 2),\\
\frac{1}{\alpha_1+\alpha_2+1}, & \mbox{otherwise}.
\end{cases}
\end{align*}
On the other hand, 
\begin{align*}
\|W_1e_{(1, 0)}\|^2- \|W^*_1e_{(1, 0)}\|^2=
\|W_2e_{(0, 1)}\|^2- \|W^*_2e_{(0, 1)}\|^2 =-\frac{1}{6}.
\end{align*}
Thus, ${\bf W}$ is sum-hyponormal, whereas neither $W_1$ nor $W_2$ is hyponormal.

We now show that ${\bf W}$ is not normaloid. To this end, let $\mathscr{H}_\kappa$ denote the reproducing kernel Hilbert space of holomorphic functions on the unit ball $\mathbb B^2 \subseteq \mathbb C^2$ with reproducing kernel
\begin{equation*}
\kappa(z,w) = \frac{3}{1-\langle z,w\rangle}- 2-\frac{3}{2}\langle z,w\rangle-\frac{3}{2}z_1z_2\bar{w}_1 \bar{w}_2, ~~ z=(z_1,z_2), \, w=(w_1,w_2) \in \mathbb B^2,
\end{equation*}
where $\langle z,w\rangle = z_1\bar{w}_1 + z_2\bar{w}_2$.
Consider the map $U$ defined by $U(e_{\alpha})=\frac{z^{\alpha}}{\|z^{\alpha}\|},$ $\alpha \in \mathbb Z^2_+$. 
By \cite[Theorem~4.14]{PR2016}, $\{z^{\alpha}\}_{\alpha \in \mathbb Z^2_+}$ forms an orthogonal basis of $\mathscr H_\kappa$. This, combined with \cite[Exercise~3.7 and Proposition 4.11]{PR2016}, yields    
\[\|z^{\alpha}\|^2 = \begin{cases}
1, & \mbox{if}~\alpha=(0,0),\\ 
\frac{2}{3}, & \mbox{if}~\alpha=(1,0),(0,1),\\ 
\frac{2}{9}, & \mbox{if}~\alpha =(1,1),\\
\frac{1}{3} \frac{\alpha_1!\alpha_2!}{(\alpha_1+\alpha_2)|!},& \mbox{if}~\alpha=(\alpha_1, \alpha_2) \in \mathbb Z_+^2 \setminus \{(1,1)\}, 
~\alpha_1+\alpha_2 \Ge 2.
\end{cases}\]
It follows that $U$ extends to a unitary map from $\mathcal H$ onto $\mathscr H_\kappa$ satisfying $$UW_j = \mathscr M_{z_j}U,\quad j=1, 2,$$  
where $\mathscr M_{z} =(\mathscr M_{z_1},\mathscr M_{z_2})$ denotes the multiplication $2$-tuple on $\mathscr{H}_\kappa$.
Thus, $\mathscr M_z$ is unitarily equivalent to the unilateral weighted $2$-shift ${\bf W}$.
A straightforward application of the spectral radius formula (see  \cite[Theorem~1]{CZ1992} and \cite[Theorem~1]{MS1992}) yields $r({\bf W})=r(\mathscr M_z)=1.$ On the other hand, $$\|W_1e_{(0, 0)}\|^2 + \|W_2e_{(0, 0)}\|^2=\frac{4}{3}.$$ Hence, ${\bf W}$ is not normaloid.
\hfill $\diamondsuit$
\end{example}

The following application of Theorem~\ref{main-4} provides sufficient conditions for a sum-hyponormal $2$-tuple of compact operators to be normal.
\begin{proposition} \label{mixed-cohypo-coro}
Let ${\bf T}=(T_1,T_2)$ be a sum-hyponormal $2$-tuple consisting of compact operators. 
Suppose that one of the following conditions holds$:$
\begin{enumerate}
\item[(a)] For some $1 \Le j \neq k \Le 2,$ 
\begin{equation}
\label{mixed-hypo}
T^*_j[T^*_k, T_k]T_j \Le 0.
\end{equation} 
\item[(b)] $\Re(T_1^*T_1T_2^*T_2)=0$.
\end{enumerate}
Then ${\bf T}$ is normal. 
\end{proposition}
\begin{proof} 
Let $\mathcal M$ be as defined in \eqref{ortho-M}. By Theorem~\ref{main-4}, ${\bf T}$ admits the decomposition
\begin{equation}
\label{decom-coro}
{\bf T}={\bf N} \oplus {\bf S}~\mbox{on~}\mathcal H = \mathcal M \oplus (\mathcal H \ominus \mathcal M),
\end{equation}
where ${\bf N}$ is a normal $2$-tuple on $\mathcal M$ and ${\bf S}=(S_1, S_2)$ is a quasinilpotent sum-hyponormal $2$-tuple on $\mathcal H \ominus \mathcal M$. 

(a) Since sum-hyponormality is invariant under permutations, we may assume that \eqref{mixed-hypo} holds for $j=1$ and $k=2$.
Note that $S^*_1[{\bf S}^*, {\bf S}]S_1 \Ge 0,$ and hence $$(S^*_1S_1)^2 + S^*_1S_2S^*_2S_1 \Le S^{*2}_1S^2_1 + S^*_1S^*_2S_2S_1.$$
Since ${\bf S}$ also satisfies \eqref{mixed-hypo}, it follows that $$(S^*_1S_1)^2 - S^{*2}_1S^2_1 \Le S^*_1[S^*_2, S_2]S_1 \Le 0.$$ Hence, by \eqref{special-appendix}, $\|S_1\|=r(S_1)$. Since $\sigma({\bf S})=\{0\},$ 
combining with the projection property for the Taylor spectrum (see \cite[Lemma~3.1]{T1970}), this implies that 
$S_1=0$. Once again, by the sum-hyponormality of ${\bf S},$ we conclude that $S_2$ is hyponormal, and a similar argument shows that $S_2=0$. Hence, by \eqref{decom-coro}, ${\bf T}={\bf N}\oplus {\bf 0}$, which completes the proof in this case.

(b) 
Since $\Re(T_1^*T_1T_2^*T_2)=0$,
we obtain 
\begin{equation}
\label{tag-1}
\Re(S_1^*S_1S_2^*S_2)=0.
\end{equation}
Let $\varphi_{\bf S}$ be as defined in \eqref{varphi-T}.
Then, by the sum-hyponormality of ${\bf S}$,
\begin{align*}
\varphi^2_{\bf S}(I)
&=
S_1^*(S_1^*S_1+S_2^*S_2)S_1
+
S_2^*(S_1^*S_1+S_2^*S_2)S_2
\\
&\Ge
S_1^*(S_1S_1^*+S_2S_2^*)S_1
+
S_2^*(S_1S_1^*+S_2S_2^*)S_2.
\end{align*}
It follows that 
\begin{align*}
\varphi^2_{\bf S}(I)- \varphi_{\bf S}(I)^2
&\Ge
(S_2^*S_1)^*S_2^*S_1
+
(S_1^*S_2)^*S_1^*S_2
-
2\Re(S_1^*S_1S_2^*S_2).
\end{align*}
By \eqref{tag-1},
$\varphi_{\bf S}(I)^2 \Le \varphi^2_{\bf S}(I)$.
Once again, by \eqref{special-appendix},
$r({\bf S})
=
\|\varphi_{\bf S}(I)\|^{1/2}.$
Since ${\bf S}$ is quasinilpotent,
$\varphi_{\bf S}(I)=0$.
Therefore,
$S_1^*S_1+S_2^*S_2=0,$
which implies
${\bf S}={\bf 0}$.
Hence, by \eqref{decom-coro}, ${\bf T}$ is normal.
\end{proof}

Note that \eqref{Putnam-d-var} is 
nontrivially optimal. Indeed, let $T \in \mathcal B(\mathcal H)$ be a hyponormal operator satisfying $\pi\|[T^*, T]\|=\mathrm{m}_2(\sigma(T))$ (for example, the unilateral unweighted shift; see \cite[Theorem~5]{P1972} for a model of such operators having real parts with simple spectra). Then the commuting pair ${\bf T}=(T, T)$ satisfies \eqref{Im-commute-new}. Moreover,
\begin{align} \notag
\pi\|[{\bf T}^*, {\bf T}]\| &= 2\pi \|[T^*, T]\|
=|1+i|^2\mathrm{m}_2\big(\sigma\big(T)\big)\big) \\ \label{optimal}
&=\mathrm{m}_2\big((1+i)\sigma\big(T)\big)\big)=\mathrm{m}_2\big(\sigma\big(T+iT)\big)\big).
\end{align}   
Although \eqref{Putnam-d-var} applies to the broad class of sum-hyponormal tuples with commuting imaginary parts, the following example demonstrates that it does not always provide the best possible estimate. 
\begin{example} Let $\mathcal H$ be a nonzero complex Hilbert space and let $A, B \in \mathcal B(\mathcal H)$ be non-normal hyponormal operators. Let $T_1 = A \otimes I$ and $T_2=I \otimes B$, and 
consider the doubly commuting pair ${\bf T}=(T_1, T_2)$ of hyponormal operators. By \cite[Theorem~2.2]{CV1978}, $\sigma({\bf T})=\sigma(T_1)\times \sigma(T_2).$
Then $\sigma(T_1+iT_2)=\sigma(T_1)+\sigma(iT_2)$. Indeed,
$$\{z_1+iz_2 : (z_1, z_2)\in \sigma({\bf T})\}=\{z_1+iz_2 : z_1 \in \sigma(T_1), z_2 \in \sigma(T_2)\}.$$
By the Brunn--Minkowski inequality (see \cite[ Theorem~12.2.2]{Ma2002}),
\begin{equation}
\label{classical-Putnam}
\mathrm{m}_2\big(\sigma(T_1+iT_2)\big)\Ge \Big(  \mathrm{m}_2\big(\sigma(T_1)\big)^{1/2}+\mathrm{m}_2\big(\sigma(T_2)\big)^{1/2}\Big)^{2}.
\end{equation}
Since $\|[{\bf T}^*, {\bf T}]\| \Le \|[T^*_1, T_1]\|+\|[T^*_2, T_2]\|$, the Putnam's inequality yields
\begin{align*}
\|[{\bf T}^*, {\bf T}]\| & \Le \, \frac{1}{\pi}\Big(\mathrm{m}_2(\sigma(T_1))+\mathrm{m}_2(\sigma(T_2))\big) \\
&\Le \, \frac{1}{\pi}\Big(\mathrm{m}_2\big(\sigma(T_1)\big)^{1/2}+\mathrm{m}_2\big(\sigma(T_2)\big)^{1/2}\Big)^{2}\\
&\stackrel{\mathclap{\eqref{classical-Putnam}}}{\Le}  \, \frac{1}{\pi}\mathrm{m}_2\big(\sigma(T_1+iT_2)\big),
\end{align*}
where the second inequality is strict because both $T_1, T_2$ are non-normal.
\hfill $\diamondsuit$
\end{example}

We now present two applications of Theorem~\ref{new-thm-Putnam} (cf. \cite[Theorem~3]{DY1992}). 
\begin{proposition} 
Let ${\bf T}=(T_1, \ldots, T_d)$ be a sum-hyponormal $d$-tuple on a complex Hilbert space $\mathcal H$ satisfying
\eqref{Im-commute-new}.
If $\sigma({\bf T})$ $($resp. $\sigma_e({\bf T}))$ is contained in 
\begin{align*}
\begin{cases} \{(z_1, \ldots, z_d) \in \mathbb C^d: z_{2j-1}, iz_{2j}  \in \mathbb R,~j=1,  \ldots, (d-1)/2, ~ z_d \in \mathbb R\} & \mbox{if}~d ~\mbox{is odd},\\ 
\{(z_1, \ldots, z_d) \in \mathbb C^d: z_{2j-1}, iz_{2j}  \in \mathbb R,~j=1,  \ldots, d/2\} & \mbox{if}~d ~\mbox{is even},
\end{cases}
\end{align*}
then ${\bf T}$ is normal $($resp. essentially normal$)$.
\end{proposition}
\begin{proof} The first assertion follows by taking $E=\mathbb R$ in Theorem~\ref{new-thm-Putnam}. To see the remaining assertion, consider the $d$-tuple
$
\mathbf{S}:=(\pi \circ q(T_1), \ldots, \pi \circ q(T_d))$ (see Remark~\ref{scalar-cpt}),
and note that by the multi-variable Atkinson theorem (see \cite[Theorem~2]{Cu1981}), $\sigma(\mathbf{S})=\sigma_e(\mathbf{T}).$ 
One may now apply the first part to $\mathbf{S}$ to complete the proof.   
\end{proof}

The following application of Theorem~\ref{new-thm-Putnam} answers Question~\ref{Q1}(iii).
\begin{proposition} 
\label{main-3-coro}
Let ${\bf T}$ be a commuting $d$-tuple on a complex Hilbert space $\mathcal H$ satisfying \eqref{Im-commute-new}. Suppose that any one of the following conditions holds$:$
\begin{enumerate}
\item[(a)] ${\bf T}$ is a sum-normal $d$-tuple.
\item[(b)] ${\bf T}$ is a sum-hyponormal $d$-tuple such that $\Im(T_j) \in \mathcal C(\mathcal H)$ for all $1 \Le j \Le d.$
\end{enumerate}
Then ${\bf T}$ is normal. 
\end{proposition}
\begin{proof}
(a) Let ${\bf S}$ be as defined in \eqref{2k-to-general}. Applying Lemma~\ref{coro-Putnam} to ${\bf S}$, we conclude that ${\bf S}$ is normal. Consequently, ${\bf T}$ is normal.

(b) The proof is an adaptation of the argument given in \cite{W1975} to the present situation. 
By assumption, $\Im({\bf T})=\big(\Im(T_1), \ldots, \Im(T_d)\big)$ is a commuting $d$-tuple of compact self-adjoint operators. By the spectral theorem (cf. \cite[Theorem~4.6]{B1996}), there exists an orthonormal eigenbasis $\{e_n\}_{n \Ge 1}$ of $\Im(\bf T)$. For every $n \Ge 1,$ let $\lambda^{(n)} = (\lambda^{(n)}_{1}, \ldots, \lambda^{(n)}_{d})$ be the eigenvalue of $\Im(\bf T)$ corresponding to $e_n$. 
Since $\lambda_j^{(n)}$ is a real number and $\Re(T_j)+\lambda_j^{(n)} I$ is a normal operator, for any $j=1, \ldots, d$ and any $n \Ge 1$, we have
\begin{align} 
 \notag
\|T_je_n\| &= \|(\Re(T_j)+ i \lambda_j^{(n)})e_n\| \\ \notag
& = \|(\Re(T_j)- i \lambda^{(n)}_j)e_n\| \\ 
&= \|T^*_je_n\|. \label{member-M}
\end{align} 
Since ${\bf T}$ is sum-hyponormal, by the Cauchy--Schwarz inequality, 
\[\big|\big \langle [{\bf T^*}, {\bf T}] x,y\big \rangle \big|^2 \Le \big\langle [{\bf T^*}, {\bf T}] x,x\big\rangle \, \big\langle [{\bf T^*}, {\bf T}] y, y\big\rangle, \quad x, y \in \mathcal H.\]  It follows that \[\ker([{\bf T^*}, {\bf T}])= \big\{x \in \mathcal H: [{\bf T^*}, {\bf T}] x= 0\big\} = \big\{x \in \mathcal H: \big \langle [{\bf T^*}, {\bf T}] x, x\big \rangle=0\big\}.\] This, together with \eqref{member-M}, shows that $\ker([{\bf T^*}, {\bf T}])$ is a subspace of $\mathcal H$ containing $\{e_n : n \Ge 1\}$. Therefore, $\ker([{\bf T^*}, {\bf T}])$ is dense in $\mathcal H$. 
Since $\ker([{\bf T^*}, {\bf T}])$ is always closed in $\mathcal H$, it follows that $\ker([{\bf T^*}, {\bf T}])=\mathcal H$. Consequently, ${\bf T}$ is sum-normal, and part (a) applies.
\end{proof}

The following is immediate from Theorem~\ref{main-3} and Remark~\ref{rem-1}(a).
\begin{proposition}
\label{sum-normal-to-normal}
 If ${\bf T}$ is a sum-normal $d$-tuple on $\mathcal H$ satisfying \eqref{double-commutator}, 
then ${\bf T}$ is normal. 
\end{proposition}
\begin{remark} 
Proposition~\ref{sum-normal-to-normal} can also be deduced from a  special case of Putnam--Kleinecke--Shirokov Theorem (see \cite[pp.~129-130]{P1954}, \cite[Problem~232]{H1967}). Indeed, since \eqref{double-commutator} holds, by the sum-normality of ${\bf T}$, each $T_j$ commutes with its self-commutator $[T^*_j, T_j]$, for $j=1, \ldots, d$. It then follows from \cite[Corollary, pp.~929--930]{P1954} that ${\bf T}$ is normal. 
\hfill $\diamondsuit$
\end{remark}

We conclude the paper with 
another application 
and some concluding remarks. Observe that if ${\bf T}$ is a commuting $d$-tuple on $\mathcal H$ satisfying $[{\bf T}^*, {\bf T}] \in \mathcal C(\mathcal H)$, then 
the $d$-tuple ${\bf S}=(S_1, \ldots, S_d)$ defined by 
$S_j=\pi \circ q(T_j),$ $j=1, \ldots, d$, is sum-normal, where $\pi$ and $q$ are as in Remark~\ref{scalar-cpt}. 
This observation yields the following corollary (cf. \cite[Lemma~2.11]{MPS2022}).
\begin{corollary} \label{coro-single}
Let ${\bf T}$ be a commuting $d$-tuple on $\mathcal H$ satisfying $[{\bf T}^*, {\bf T}] \in \mathcal C(\mathcal H)$. Then the following statements are equivalent$:$
\begin{enumerate}
\item[(i)] ${\bf T}$ is essentially normal.
\item[(ii)] $[T^*_j, T_k] \in {\mathcal C}(\mathcal H)$ for all $1 \Le j \neq k \Le d$. 
\item[(iii)] $\big[T_j, [T^*_k, T_k]\big] \in {\mathcal C}(\mathcal H)$ for all $1 \Le j  \neq k \Le d$. 
\end{enumerate}  
\end{corollary}

\begin{proof} The implication (i)$\Rightarrow$(ii) follows from the Fuglede--Putnam theorem (see \cite[Corollary~1.18]{RR1973}), whereas (ii)$\Rightarrow$(i) follows from Proposition~\ref{main-3-coro}(a) and \eqref{d-comm-formula-im}. Finally, the implication (i)$\Rightarrow$(iii) is immediate, while the implication (iii)$\Rightarrow$(i) follows from Proposition~\ref{sum-normal-to-normal}.
\end{proof}

It is worth noting that Corollary~\ref{coro-single} applies to sum-normal $d$-tuples. However, the equivalence of (i)--(iii) no longer holds if sum-normality is replaced by sum-hyponormality. Indeed, for every integer $d \Ge 2$, the multiplication $d$-tuple $\mathscr M_z$ on the Hardy space $H^2(\mathbb D^d)$ over the unit polydisc $\mathbb D^d$ is a doubly commuting hyponormal $d$-tuple and hence serves as a counterexample (see \cite[Section~1.4.3]{PR2016} for the definition of $H^2(\mathbb D^d)$). In contrast, the existence of a nonzero quasinilpotent sum-normal $d$-tuple, or more generally, a nonnormal sum-normal $d$-tuple, remains unknown (see Question~\ref{Q1}(ii)). In this context, a result of Douglas asserts that every compact quasinilpotent operator is the limit of nilpotent operators (see \cite[p.~916]{H1970}). In light of Theorem~\ref{nilpotent-thm}, this result suggests that there may be no nonzero compact quasinilpotent sum-normal $d$-tuples.

\appendix

\section{Norm and spectral radius of completely positive maps}

Let $\mathcal{A}$ be a unital $C^*$-algebra. For a positive integer $n$, 
let $M_n(\mathcal{A})$ denote the $C^*$-algebra of all $n \times n$
matrices with entries in $\mathcal{A}$. 
If $\varphi : \mathcal{A} \rightarrow \mathcal{A}$ is a linear map,
then, for each positive integer $n$, we define $\varphi_n : M_n(\mathcal{A}) \rightarrow M_n(\mathcal{A})$ by 
\begin{equation*}
 \varphi_n\big((a_{j,k})_{1 \Le j,k \Le n}\big):=\big(\varphi(a_{j,k})\big)_{1 \Le j,k \Le n}, 
\end{equation*}
where $(a_{j,k})_{1 \Le j,k \Le n} \in M_n(\mathcal{A})$.
The map $\varphi$ is said to be \textit{completely positive} (or simply, a {\it CP map}) if $\varphi_n$ is positive
for every positive integer $n$. 

The following fundamental theorem, due to Stinespring (see \cite[Theorem 4.1]{P2003}), characterizes
completely positive maps from a $C^*$-algebra into $\mathcal{B}(\mathcal{H})$.
\begin{SDT}
Let $\mathcal{A}$ be a unital $C^*$-algebra, and let $\varphi:\mathcal{A} \to \mathcal{B}(\mathcal H)$ be a completely positive map. Then there exist a Hilbert space $\mathcal K$, a unital $*$-homomorphism $\pi :\mathcal{A} \to \mathcal{B}(\mathcal K)$, and a bounded linear transformation $V : \mathcal H \rightarrow \mathcal K$ such that \[\varphi(a) = V^*\pi(a)V, \ a \in \mathcal{A}.\]
Moreover, $\left\|\varphi(1)\right\|=\left\|V\right\|^2$.
Furthermore, $\pi$, $V$, and $\mathcal{K}$ may be chosen so that $\mathcal{K}=\pi(\mathcal{A})V\mathcal{H}$.
\end{SDT}
When $\mathcal{K}=\pi(\mathcal{A})V\mathcal{H}$, the triple $(\pi,V,\mathcal{K})$ is called as a \textit{minimal Stinespring representation of $\varphi$.}

\begin{PropA} 
Let $\varphi:\mathcal{B}(\mathcal H) \to \mathcal{B}(\mathcal H)$ be a CP map such that 
\begin{equation} \label{submultiplicativity}
\varphi(I)^2 \Le \varphi^2(I). 
\end{equation}
Then $\displaystyle \lim_{k \to \infty} \|\varphi^k(I)\|^{\frac{1}{k}}=\|\varphi(I)\|.$
\end{PropA}
\begin{proof}
Since $\varphi$ is a CP map, so is $\varphi^k$ for every positive integer $k$. Let $(\pi_k, V_k, \mathcal K_k)$ denote the minimal Stinespring representation of $\varphi^k$, $k \Ge 1$. Fix $n \in \mathbb N$ and $x \in \mathcal H$. Then
	\begin{align}\label{spradius1}
	\nonumber\langle \varphi^n(I)x, x \rangle^2 &=\langle \varphi^{n-1}(\varphi(I))x,x\rangle^2\\ \nonumber 
&=\langle \pi_{n-1}(\varphi(I))V_{n-1}x,V_{n-1}x\rangle^2\\ &\Le \|\pi_{n-1}(\varphi(I)) V_{n-1} x\|^2 \| V_{n-1} x\|^2,
	\end{align}
where the last inequality follows from the Cauchy--Schwarz inequality. 
Further, since
$\pi_{n-1}$ is a $*$-homomorphism, it follows from
\eqref{submultiplicativity} that
\begin{align}\label{spradius2}
	\nonumber\|\pi_{n-1}(\varphi(I)) V_{n-1} x\|^2 
&=\langle \pi_{n-1}(\varphi(I)) V_{n-1} x,\pi_{n-1}(\varphi(I)) V_{n-1} x \rangle  \\& \nonumber = \langle V^*_{n-1}\pi_{n-1}(\varphi(I))^*\pi_{n-1}(\varphi(I)) V_{n-1} x, x\rangle 
\\& \nonumber = \langle V^*_{n-1}\pi_{n-1}(\varphi(I)^2) V_{n-1} x, x\rangle 
\\
&\Le \nonumber \langle V^*_{n-1}\pi_{n-1}(\varphi^2(I))V_{n-1}x,x\rangle \\&= \langle \varphi^{n+1}(I)
	x,x\rangle. \end{align} 
Moreover, since $\pi_k(I)=I,$ $k \Ge 1$, 
$$\|V_{n-1}x\|^2=\langle V^*_{n-1}V_{n-1}x,x\rangle = \langle \varphi^{n-1}(I)x,x\rangle.$$ 
This, together with \eqref{spradius1} and \eqref{spradius2}, yields
	\begin{equation*}
	\langle \varphi^n(I)x, x \rangle^2 \Le \langle \varphi^{n+1}(I)x, x \rangle \langle \varphi^{n-1}(I)x, x \rangle. \end{equation*}
Since $\varphi^k(I) \Ge 0,$ $k \Ge 1$, it follows that \begin{equation} \label{norminequality}
	\|\varphi^n(I)\|^2 \Le \|\varphi^{n+1}(I)\|\|\varphi^{n-1}(I)\|.
	\end{equation}
By the general theory of log-convex sequences, $\lim_{n \rightarrow \infty}\|\varphi^n(I)\|^{1/n}$ exists (possibly $+\infty$).
For a positive integer $n$, define \[M_{k, n} := \frac{\|\varphi^{n-k}(I)\|^{k+2}}{\|\varphi^{n-k-1}(I)\|^{k+1}}, \quad k= 0, \ldots, n-1.\] Then, for $0 \Le k \Le n-2,$ 
\begin{align*}
M_{k+1, n} &=
\, \left(\frac{\|\varphi^{n-k-1}(I)\|^{2}}{\|\varphi^{n-k-2}(I)\|}\right)^{k+2}  \frac{1}{\|\varphi^{n-k-1}(I)\|^{k+1}} \\
&\stackrel{\mathclap{\eqref{norminequality}}}{\Le} \, \frac{\|\varphi^{n-k}(I)\|^{k+2}}{\|\varphi^{n-k-1}(I)\|^{k+1}} = M_{k, n}.
\end{align*} 
Therefore, the sequence $\{M_{k, n}\}_{0 \Le k \Le n-1}$ is decreasing. Consequently, 
\begin{align*}
\|\varphi(I)\|^{n+1} &= M_{n-1, n} \Le M_{n-2, n} \Le \cdots \Le M_{0, n} 
= \frac{\|\varphi^{n}(I)\|^2}{\|\varphi^{n-1}(I)\|} \overset{\eqref{norminequality}}\Le \|\varphi^{n+1}(I)\|.
\end{align*}
Thus, $\|\varphi(I)\| \Le \|\varphi^k(I)\|^{1/k}$ for any positive integer $k$. 
This shows that 
\begin{equation} \label{formula-near}
 \|\varphi(I)\| \Le \lim_{k \to \infty} \|\varphi^k(I)\|^{1/k}. 
\end{equation}
On the other hand, the Russo--Dye Theorem (see \cite[Corollary 2.9]{P2003}) asserts that $\|\psi\| =\|\psi(I)\|$ for every completely positive map $\psi$. Combining this with \eqref{formula-near} and the inequality $\|\varphi^k\| \Le \|\varphi\|^k,$ $k \in \mathbb N,$ yields the desired formula. 
\end{proof}
Let $\varphi_{\bf T}$ be as defined in \eqref{varphi-T}.
If ${\bf T}$ is hyponormal, then by the proof of \cite[Lemma~3.10]{CS2017},  
$\varphi_{\bf T}(I)^2 \Le \varphi^2_{\bf T}(I)$.
Applying Proposition~A to $\varphi_{\bf T}$ as defined above, we conclude from the spectral radius formula (see  \cite[Theorem~1]{CZ1992} and \cite[Theorem~1]{MS1992}) that ${\bf T}$ is normaloid. This recovers the implication \eqref{special-appendix} and \cite[Lemma~3.10]{CS2017} (see \cite[Theorem~3.5]{S2025} for a generalization). 

\vskip.2cm
\noindent 
\textit{Statement and Declarations:}

\vskip.3cm

\noindent
{\bf Conflict of interest} The authors declare that they have no conflict of interest.

\vskip.3cm

\noindent 
{\bf Data Availability} No data was used for the research described in the article.


\begin{thebibliography}{100}
  
\bibitem{A1963} T. Andô,
On hyponormal operators, {\it Proc. Amer. Math. Soc.} {\bf 14} (1963), 290--291.

\bibitem{At1988} A. Athavale, On joint hyponormality of operators, {\it Proc. Amer. Math. Soc.} {\bf 103} (1988) 417--423. 
  
\bibitem{B1996} T. Bhattacharyya, On tuples of commuting compact operators, {\it Publ. Res. Inst. Math. Sci.} {\bf 32} (1996), 785--795.

\bibitem{B1971} J. Bunce,
The joint spectrum of commuting nonnormal operators,
{\it Proc. Amer. Math. Soc.} {\bf 29} (1971), 499--505.

\bibitem{CV1978}  Z. Ceauşescu, F.-H. Vasilescu,  Tensor products and the joint spectrum in Hilbert spaces, {\it Proc. Amer. Math. Soc.} {\bf 72} (1978), 505--508.

\bibitem{CS2017} S. Chavan, V. M. Sholapurkar, Completely hyperexpansive tuples of finite order, {\it J. Math. Anal. Appl.} {\bf 447} (2017), 1009--1026.	

\bibitem{CCHZ2000} M. Ch$\bar{\mbox{o}}$, R. E. Curto, T. Huruya, W. $\dot{\mbox{Z}}$elazko, 
Cartesian form of Putnam's inequality for doubly commuting hyponormal $n$-tuples, 
{\it Indiana Univ. Math. J.} {\bf 49} (2000), 1437--1448.

\bibitem{CT1979} M. Ch$\bar{\mbox{o}}$, M. Takaguchi, Joint spectra of matrices,
{\it Sci. Rep. Hirosaki Univ.} {\bf 26} (1979), 15--19.

\bibitem{CT1984} M. Ch$\bar{\mbox{o}}$, M. Takaguchi,
Some classes of commuting $n$-tuples of operators,
{\it Studia Math.} {\bf 80} (1984), 245--259.

\bibitem{CZ1992} M. Ch$\bar{\mbox{o}}$, W. $\dot{\mbox{Z}}$elazko,
On the geometric radius of commuting $n$-tuples of operators, {\it Hokkaido Math. J.} {\bf 21} (1992), 251--258. 
	 
\bibitem{Co1990} J. B. Conway, {\it A course in functional analysis},
Second edition. Graduate Texts in Mathematics, 96, Springer-Verlag, New York, 1990. xvi+399 pp.


\bibitem{Cu1981}  R. E. Curto, Fredholm and invertible $n$-tuples of operators. The deformation problem,
{\it Trans. Amer. Math. Soc.} {\bf 266} (1981), 129--159.
	 
\bibitem{Cu1990} R. E. Curto, 
Joint hyponormality: a bridge between hyponormality and subnormality. Operator theory: operator algebras and applications, Part 2 (Durham, NH, 1988), 69--91,
{\it Proc. Sympos. Pure Math.}, {\bf 51}, Part 2, Amer. Math. Soc., Providence, RI, 1990. 

\bibitem{DY1992} R. G. Douglas, K. Yan, A multi-variable Berger--Shaw theorem, {\it J. Oper. Theory} {\bf 27} (1992), 205--217. 


\bibitem{DK2000} J. J. Duistermaat, J. A. C. Kolk, {\it 
Lie groups}, 
Universitext. Springer-Verlag, Berlin, 2000. viii+344 pp. 
	 
\bibitem{EP2001} J. Eschmeier, M. Putinar, Some remarks on spherical isometries, Systems, approximation, singular integral operators, and related topics (Bordeaux, 2000), 271--291,
{\it Oper. Theory Adv. Appl.,} {\bf 129}, Birkhäuser, Basel, 2001.

\bibitem{H1967} P. R. Halmos, {\it A Hilbert space problem book}, D. Van Nostrand Co., Inc., Princeton, N.J.-Toronto, Ont.-London, 1967. xvii+365 pp.

\bibitem{H1970} P. R. Halmos, 
Ten problems in Hilbert space,
{\it Bull. Amer. Math. Soc.} {\bf 76} (1970), 887--933. 

\bibitem{HK2010} Y. Han, A. Kim, Weyl's theorem in several variables, {\it J. Math. Anal. Appl.} {\bf 370} (2010),
538--542.

\bibitem{H2023} M. Hartz, 
{\it An invitation to the Drury-Arveson space}, Lectures on analytic function spaces and their applications, 347--413,
Fields Inst. Monogr., {\bf 39}, Springer, Cham, 2023.

\bibitem{II2006} K. J. Izuchi, K. H. Izuchi, Rank-one commutators on invariant subspaces of the Hardy space on the
bidisk, {\it J. Math. Anal. Appl.} {\bf 316} (2006), 1--8.

\bibitem{IO1994} K. J. Izuchi, S. Ohno, Selfadjoint commutators and invariant subspaces on the torus, {\it J. Oper. Theory} {\bf 31} (1994), 189--204. 

\bibitem{M2002} M. Martin, Self-commutator inequalities in higher dimension, 
{\it Proc. Amer. Math. Soc.} {\bf 130} (2002), 2971--2983.

\bibitem{MP1989} M. Martin, M. Putinar, {\it Lectures on hyponormal operators}, 
Operator Theory: Advances and Applications, 39. Birkhäuser Verlag, Basel, 1989. 304 pp. 

\bibitem{Ma2002} J. Matoušek, 
{\it Lectures on discrete geometry},
Graduate Texts in Mathematics, 212. Springer-Verlag, New York, 2002. xvi+481 pp.

\bibitem{MPS2022} G. Misra, P. Pramanick, K. B. Sinha, 
A trace inequality for commuting $d$-tuples of operators,  {\it Integral Equations Operator Theory} {\bf 94} (2022), Paper No. 16, 37 pp.

\bibitem{MS1992} V. M$\ddot{\mbox{u}}$ller, A. Soltysiak, Spectral radius formula for commuting Hilbert space operators,
{\it Studia Math.} {\bf 103} (1992), 329--333.

\bibitem{P2003} V. Paulsen, \textit{Completely bounded maps and operator algebras}, Cambridge University Press, Cambridge 2003.

\bibitem{PR2016} V. I. Paulsen, M. Raghupathi, {\it An introduction to the theory of reproducing kernel Hilbert spaces}, Cambridge Studies in Advanced Mathematics, {\bf 152}, Cambridge University Press, Cambridge, 2016. 

\bibitem{P1954} C. R. Putnam, On the spectra of commutators, 
{\it Proc. Amer. Math. Soc.}, {\bf 5} (1954), 929--931.

\bibitem{P1970} C. R. Putnam, 
An inequality for the area of hyponormal spectra,
{\it Math. Z.} {\bf 116} (1970), 323--330.

\bibitem{P1972} C. R. Putnam, Hyponormal operators having real parts with simple spectra, {\it Trans. Amer. Math. Soc.} {\bf 172} (1972), 447--464.

\bibitem{R1966} H. Radjavi, 
Structure of $A^*A - AA^*$, {\it 
J. Math. Mech.} {\bf 16} (1966), 19--26.
	
	\bibitem{RR1973} H. Radjavi, P. M. Rosenthal, \textit{Invariant subspaces}, Ergebnisse der Mathematik und ihrer Grenzgebiete, Band 77, Springer, New York-Heidelberg, 1973. 

\bibitem{RR2000} H. Radjavi, P. M. Rosenthal, 
{\it Simultaneous triangularization},
Universitext. Springer-Verlag, New York, 2000, xii+318 pp. 

\bibitem{S2015} B. Simon, {\it Operator theory. A Comprehensive Course in Analysis}, Part 4. American Mathematical Society, Providence, RI, 2015, xviii+749 pp.

\bibitem{S1962} J. G. Stampfli, 
Hyponormal operators,
{\it Pacific J. Math.} {\bf 12} (1962), 1453--1458.


	
\bibitem{S2025} H. Stanković, Joint paranormality: bridging the gap between hyponormal and normaloid tuples, 
{\it Linear Multilinear Algebra}, {\bf 73} (2025), 3004--3018.


\bibitem{T1970} J. L. Taylor, A joint spectrum for several commuting operators, {\it J. Funct. Anal.} {\bf 6} (1970),
172--191.

\bibitem{T1970b} J. L. Taylor,
The analytic-functional calculus for several commuting operators, {\it Acta Math.} {\bf 125} (1970), 1--38.

\bibitem{W1975} R. Whitley, 
A note on hyponormal operators,
{\it Proc. Amer. Math. Soc.} {\bf 49} (1975), 399--400. 

\bibitem{X1983} D. Xia, On the semi-hyponormal $n$-tuple of operators, {\it Integral Equations Operator Theory} {\bf 6} (1983), 879--898.

\end{thebibliography}
\end{document}